\documentclass[12pt]{amsart}

\usepackage{amsmath, amssymb, amscd, newlfont}

\newcommand{\Hom}{{\mathrm{Hom}}}

\newcommand{\Ann}{{\mathrm{Ann}}}
\newcommand{\Irr}{{\mathrm{Irr}}}

\newcommand{\supp}{{\mathrm{supp}}}

\newcommand{\Ind}{{\mathrm{Ind}}}

\newcommand{\calO}{{\mathcal{O}}}

\usepackage[mathscr]{eucal}

\numberwithin{equation}{section}

\newtheorem{Prop}[equation]{Proposition}
\newtheorem{Lem}[equation]{Lemma}
\newtheorem{Def}[equation]{Definition}
\newtheorem{Thm}[equation] {Theorem}
\newtheorem{Cor}[equation]{Corollary}

\title
  {On Ideals Defining  Irreducible Representations of  Reductive $p$--adic Groups}
\author{ Goran Mui\' c}

\address{ Department of Mathematics, Faculty of Science,
University of Zagreb,
Bijeni\v cka 30, 10000 Zagreb,
Croatia}
\email{gmuic@math.hr}
\thanks{The  author acknowledges Croatian Science Foundation grant IP-2018-01-3628.}

\begin{document}

\begin{abstract}
  Let $G$ be a reductive $p$--adic group. Assume that $L\subset G$ is an open--compact subgroup, and $\mathcal H_L$
   is the Hecke algebra of $L$--biinvariant complex functions on $G$.  It is a well--known and standard result about how to prove existence of
   a complex smooth irreducible $G$--module  out of a maximal left ideal $I\subset \mathcal H_L$. Using theory about Bernstein center we make this construction explicit.
   This leads us to some very interesting questions. 
   \end{abstract} 
\maketitle

\section{Introduction}\label{intr}

 Let $k$ be a non--Archimedean local field. Let $\mathcal O\subset k$ be its ring of integers, and
 let $\varpi$ be  a generator of the maximal ideal in $\calO$. Let $q$ be the number of elements in the residue field $\mathcal O/ \varpi\mathcal O$.
 Assume that is $G$  is a   Zariski connected reductive group defined over $k$. By abuse of notation, we write $G$ for the group of $k$--points. Similarly we do for algebraic subgroups defined
 over $k$. We fix a  Haar measure $dy$ on $G$, and consider  $\mathcal H=C_c^\infty(G)$  as an associative algebra under the convolution defined with respect to $dy$.
 For any open compact subgroup $L\subset G$, we let  $\mathcal H_L$ be the subalgebra of $\mathcal H$ consisting of all $L$--binvaiant functions in  $\mathcal H$. This is an associative algebra
 with identity $\epsilon_L$ (see (\ref{prelim-00})).

 In this paper maximal ideals are assumed to be proper. We will be concerned with the following simple and well--known result:
 Let $W$ be a (finite--dimensional) irreducible unital $\mathcal H_L$--module,
then there exists a unique up to an isomorphism irreducible smooth $G$--module $V$ such that $V^L$ is isomorphic to $W$ as a $\mathcal H_L$--modules
(see \cite{BZ}, Proposition 2.10 c)). All subsequent proofs of this result that we are able to find are not very constructive. In (\cite{muic}, Theorem 3-9) we give a very explicit construction of the
representation $V$ once we  fix a maximal left ideal in $\mathcal H_L$ such that $W\simeq \mathcal H_L/I$ (see elementary Lemma \ref{prelim-01} about relation between different possibilities for
left ideal $I$; $I$ is uniquely determined by $W$ if $\mathcal H_L$ is commutative).

So, let us fix $L\subset G$ an open compact subgroup, and let $I\subset \mathcal H_L$ be a maximal  left ideal. Following (\cite{muic}, Theorem 3-9) (see Lemma \ref{prelim-1}), we put
$$
J_{I, L}\overset{def}{=} \text{sum of all proper left ideals
  in $\mathcal H\star \epsilon_L$ which contain $\mathcal H\star I$.}
$$
Then, $J_{I, L}$ is a unique maximal left ideal in $\mathcal H\star \epsilon_L$ which contains $I$. The corresponding irreducible smooth $G$--module $ \mathcal V(I, L)$ (see Lemma \ref{prelim-1})
has space of $L$--invariants isomorphic to $\mathcal H_L/I$ as $\mathcal H_L$--modules. It is observed in (\cite{muic}, Theorem 3-9) (see Lemma \ref{prelim-1} (v)) that a smooth $G$--module
$$
  \mathcal W(I, L)\overset{def}{=} \mathcal H \star \epsilon_L/\mathcal H\star I
  $$
has a unique maximal proper subrepresentation, and the corresponding quotient is
$\mathcal V(I, L)$. The canonical projection $\mathcal W(I, L)^L\longrightarrow \mathcal V(I, L)^L$ is an isomorphism of $\mathcal H_L$--modules.

The goal of this note is to understand  the module $\mathcal W(I, L)$. We start with an important case. Assume that $L$ has an Iwahori factorization  (see \cite{bushnell}, 1.2).
Then $\mathcal W(I,  L)$ being generated by  $\mathcal W(I, L)^L$  has property that every submodule has a non--zero vector invariant under $L$ by a  rather deep result
(\cite{Be}, 3.9, see also \cite{renard}, Corollaire VI.9.4 for the proof, and \cite{bushnell}, 1.2).
This immediately implies the   result:
$$
\mathcal W(I,  L)=\mathcal V(I, L). 
$$
Hence, we have proved the following proposition:

\begin{Prop}\label{intr-m-0} Let $L\subset G$ be an open compact subgroup  having an Iwahori factorization, and  $I\subset \mathcal H_L$  be a maximal left ideal.
  Then, $\mathcal W(I,  L)=\mathcal V(I, L)$, and, consequently, $J_{I, L}=\mathcal H\star I$.
\end{Prop}

This is a particular case  of the following general theorem which is proved using techniques from (\cite{Be}, \cite{bdk}).  We remark that a hyperspecial maximal compact subgroup does
not posses an Iwahori  factorization.

\begin{Thm}\label{intr-m-1} Let $L\subset G$ be an open compact, and  $I\subset \mathcal H_L$  be a maximal left ideal. Then, $\mathcal W(I,  L)$ is an admissible smooth $G$--module of finite length,
  all of its irreducible subquotients have the same infinitesimal character (see \cite{bdk} for definition or Section \ref{bc} in this paper), 
it has  a unique maximal proper subrepresentation, and the corresponding quotient is unique up to an isomorphism irreducible smooth $G$--module which space of  $L$--invariants is
  isomorphic to $\mathcal H_L/I$ as a $\mathcal H_L$--module.
\end{Thm}

In Section \ref{prelim}, we consider  elementary theory of Hecke algebras and its ideals. In Section \ref{bc}, we recall from \cite{bdk} some results
about the Bernstein center  needed in the proof of Theorem \ref{intr-m-1}. We also prove a few more results.
Finally, we prove Theorem \ref{intr-m-1} in Section \ref{pmr}.  Now, we discuss some of the consequences.

Let $V$ be an irreducible smooth $G$--module. For each open compact subgroup $L\subset G$ such that $V^L\neq 0$ we fix a maximal  left ideal in $\mathcal H_L$,
say $I_{V, L}$, in the equivalence class determined by $\mathcal H_L$--module $V^L$ (see Lemma \ref{prelim-01}).

The number of irreducible smooth $G$--modules with the same infinitesimal character is finite (see \cite{bdk} or
Section \ref{bc}). Therefore, for sufficiently small  open compact subgroup $L_0\subset G$, all of them have a non--zero vector  invariant under $L$ where $L$ is any
open--compact subgroup of $G$ contained in $L_0$. So, we obtain

\begin{Cor}\label{intr-m-3} Let $V$ be an irreducible smooth $G$--module. Then,
  there exists an open compact subgroup $L_0\subset G$ depending on $V$ such that for any open compact subgroup $L\subset L_0$, we have
  $\mathcal W(I_{V, L}, L)=V$.  
\end{Cor}

Finally, we state the following corollary:

\begin{Cor}\label{intr-m-4} Let $V$ be an irreducible smooth $G$--module.
If the infinitesimal character of $V$ is in general position (see Definition \ref{bc-1}),
then  $\mathcal W(I_{V, L}, L)=V$ for all $L$ such that  $V^L\neq 0$.
\end{Cor}

This follows immediately from  Theorem \ref{intr-m-1} again since $\mathcal W(I,  L)^L$ is irreducible, and  
$V$ is the only irreducible $G$--module with the infinitesimal character of $V$  (see Definition \ref{bc-1}).
We remark that the set of infinitesimal characters is an affine variety (\cite{bdk} or Section \ref{bc} here), and
there exists Zariski open subset of that variety consisting of infinitesimal characters in general position (see Lemma \ref{bc-2}).

\vskip .2in
We note that $\mathcal W(I, L)$ is not always irreducible. Before we state the result, we introduce some notation.
Let $A_0$ be a maximal $k$--split torus of $G$.  Let $P_0$ be a minimal $k$--parabolic subgroup of $G$ corresponding to some choice of positive roots of $A_0$ in $G$.
Let $U_0$ be the unipotent radical of $P_0$. Let  $K$ be a compatible hyperspecial maximal compact subgroup.  We have the following result:

\begin{Prop}\label{intr-m-5}   Assume that  $V$ is a  $K$--spherical irreducible smooth $G$--module with its data in the
  Langlands classification of irreducible representations supported on $P_0$ i.e., there exists an unramified character $\chi$ of the Levi subgroup $M_0=Z_{G}(A_0)$ of $P_0$ satisfying usual positiveness
  conditions such that  $V$ is a unique irreducible quotient  of the parabolically induced representation  $\Ind_{P_0}^G(\chi)$ (see for example \cite{renard}, Theorem VII.4.2).
  Then,
  $$\mathcal W(I_{V, K}, K) \simeq \Ind_{P_0}^G(\chi).
  $$ 
\end{Prop}

\begin{proof} Since $\Ind_{P_0}^G(\chi)^K \simeq V^K$ as a $\mathcal H_K$--module  generates whole $\Ind_{P_0}^G(\chi)$, we see that $\Ind_{P_0}^G(\chi)$ is a quotient of
  $\mathcal W(I_{V, K}, K)$. Next, by  Theorem \ref{intr-m-1}, $\mathcal W(I_{V, K}, K)$ has finite length, all of its irreducible subqoutients are among those of
  $\Ind_{P_0}^G(\chi)$.     Now, the explicit computation of the Jacquet module of $\mathcal W(I_{V, K}, K)$ with respect to $P_0$  (see Proposition \ref{cjm-8}) completes the proof. 
  \end{proof}

\vskip .2in
The induced representation $\Ind_{P_0}^G(\chi)$ can have a large number of irreducible subquotients. For example, one can see that from the Zelevinsky classification for $GL(n, k)$ (see \cite{Z}).
When $G=GL(2, k)$ and $V$ is the trivial representation, the induced representation    $\Ind_{P_0}^G(\chi)$ has length two; Steinberg representation of $GL(2, k)$ is a unique  subrepresentation.

\vskip .2in

I would like to thank Gordan Savin for  explaining to me the proof of Proposition \ref{intr-m-5} when $G$ is split and adjoint using methods of Iwahori Hecke algebras and
Kazhdan--Lusztig classification (see  \cite{KL}). I would like to thank the referee for useful  comments on the content of the paper.

\section{Elementary Properties of Hecke Algebras and Ideals} \label{prelim}

Let $k$ be a non--Archimedean local field. Let $\mathcal O\subset k$ be its ring of integers, and
  let $\varpi$ be  a generator of the maximal ideal in $\calO$. Let $q$ be the number of elements in the residue field $\mathcal O/ \varpi\mathcal O$.
  Assume that is $G$  is a   Zariski connected reductive group defined over $k$. By abuse of notation, we write $G$ for the group of $k$--points. Similarly we do for subgroups defined
  over $k$. We fix the Haar measure $dy$ on $G$ and let $\mathcal H=C_c^\infty(G)$  considered as an associative algebra under the convolution:
  $$
f\star g(x)= \int_G f(xy^{-1}) g(y) dy.
$$
For any open compact subgroup $L\subset G$, we let  $\mathcal H_L$ be the subalgebra of $\mathcal H$ consisting of all $L$--biinvariant functions in  $\mathcal H$. This is an associative algebra
with identity
\begin{equation}\label{prelim-00}
\epsilon_L\overset{def}{=}\frac{1}{vol(L)}1_L,
\end{equation}
where $1_L$ is the characteristic function of $L$ in $G$. We have
$$
\mathcal H=\cup_L \mathcal H_L,
$$
where $L$ ranges over all open compact subgroups of $G$ (or just over a basis of neighborhoods of $1$). By a standard definition (\cite{BZ}, 2.5), $\mathcal H$--module
$V$ is non--degenerate if  for any $v\in V$ there exists an
open compact subgroup $L\subset G$ such that  $\epsilon_L.v=v$. The space $\epsilon_L.V$ is an unital module for  $\mathcal H$. A non--degenerate $\mathcal H$--module gives rise to a
unique smooth $G$--module such that
$$
x.\left(f.v\right)\overset{def}{=}\left(l_xf\right).v, \ \ f\in \mathcal H(G), \ \ v\in V,
$$
where $l_x$ is the left translation $l_xf(y)=f(x^{-1}y)$.  We have
$$
V^L=\epsilon_L.V,
$$
for all open--compact subgroups $L\subset G$.

\vskip .2in
The category of all smooth  $G$--modules can be identified with the category of all non--degenerate  $\mathcal H$--modules. 
In particular,  an  $\mathcal H$--module is irreducible if and only if its is irreducible  smooth $G$--module.
By a theorem of Jacquet (\cite{renard}, Theorem VI.2.2), every irreducible smooth
$G$--module  $V$  is admissible i.e., the space $V^L$ is a finite dimensional complex vector space for all open--compact subgroups of $G$. For irreducible $V$, if $V^L\neq 0$, then it is an
irreducible $\mathcal H_L$--module (see \cite{BZ}, Proposition 2.10). Moreover, by the same reference, let  $L\subset G$ be an open--compact subgroup, and
assume that $V_i$, $i=1,2$, are irreducible   smooth $G$--modules such that  $V^L_i\neq 0$, $i=1, 2$. Then, $V_1$ is equivalent to  $V_2$ as  a $G$--module if and only if  $V^L_1$ is
equivalent to
$V^L_2$ as  a $\mathcal H_L$--module.

\vskip .2in
Assume that $W$ is an irreducible  unital $\mathcal H_L$--module. 
Let $w\in W$ be a non--zero vector in $W$.  Then
$$
I\overset{def}{=}I_w\overset{def}{=}\Ann_{\mathcal H_L}(w)
$$
is a maximal left ideal in $\mathcal H_L$. Obviously, we have
$$
\mathcal H_L /I\simeq W
$$
as $\mathcal H_L$--modules. Different choice of a vector $w'$ result in a  existence of $f, g\in \mathcal H_L$ such that  $w'=f.w$, and $w=g.w'$. Put  $I'=I'_{w'}$. Then, we have the
following:
$$ 
\begin{cases}\label{prelim-0}
  I\star g\subset I', \ \ I'\star f\subset I, \ \text{and}\\
\epsilon_L- g\star  f \equiv \ 0 \left(\text{mod} \ I\right), \ \  \epsilon_L- f\star  g \equiv \ 0 \left(\text{mod} \ I'\right).
\end{cases}
$$
\vskip .2in
 We have following lemma:

 \begin{Lem}\label{prelim-01} Let $L\subset G$ be an open compact subgroup. We define that two maximal  left ideals  $I$ and $I'$ in $\mathcal H_L$ are equivalent if and only if
   there exist
  $f, g\in \mathcal H_L$ such that (\ref{prelim-0}) holds. Then, maximal  left ideals  $I$ and $I'$ are equivalent if and only if the corresponding $\mathcal H_L$--modules are
  isomorphic:
  $$
  \mathcal H_L/ I\simeq \mathcal H_L/ I'.
  $$
  Furthermore, the classes of equivalence of irreducible unital $\mathcal H_L$--modules are parametrized by the equivalence classes of maximal left ideals in $\mathcal H_L$
  \end{Lem}
\begin{proof} We leave details to the reader.
  \end{proof}

\vskip .2in
We continue with the following simple result (see \cite{muic}, Theorem 3-9) which makes  (\cite{BZ}, Proposition 2.10 c)) more explicit.

\begin{Lem}\label{prelim-1} 
  Let $L\subset G$ be an open--compact subgroup. Then, for each maximal  left ideal  $I\subset \mathcal H_L$,
  there exists a unique  left ideal $J'$ of $\mathcal H$ such that the following three conditions hold:
\begin{itemize}
\item[(i)] $J'\subset  \mathcal H\star \epsilon_L$;
\item [(ii)]  $I\subset J'$, or equivalently $\mathcal H \star I\subset J$;
\item[(iii)] $\mathcal H\star \epsilon_L/J'$ is irreducible.
\end{itemize}
The left ideal $J'$ is a unique maximal left--ideal, denoted by
$$
J_I=J_{I, L},
$$
in  $\mathcal H\star \epsilon_L$
which contains $I$. It is a sum of all proper left ideals
in $\mathcal H\star \epsilon_L$ which contain $I$. Moreover, $\epsilon_L\star J_{I, L}= I$. 
\begin{itemize}
\item[(iv)] Regarding
$$
  \mathcal V(I, L)\overset{def}{=} \mathcal H \star \epsilon_L/J_{I, L}
  $$
as a smooth $G$--module (see above), we have that its space of $L$--invariants  is isomorphic to (the irreducible module)  $\mathcal H_L/I$ as a
   $\mathcal H_L$--module. Up to isomorphism,  $\mathcal V(I, L)$ is  unique irreducible smooth $G$--module with this property (see \cite{BZ}, Proposition 2.10). 
\item[(v)] The smooth $G$--module
  $$
  \mathcal W(I, L)\overset{def}{=} \mathcal H \star \epsilon_L/\mathcal H\star I
  $$
  has a unique maximal proper subrepresentation, and the corresponding quotient is
  $\mathcal V(I, L)$. The canonical projection $\mathcal W(I, L)^L\longrightarrow \mathcal V(I, L)^L$ is an isomorphism of $\mathcal H_L$--modules.
  The module  $\mathcal W(I, L)$ is generated by $\epsilon_L+\mathcal H\star I$. 
\end{itemize}
\end{Lem}

\section{Bernstein center}\label{bc}

In this section we describe Bernstein center and its action on smooth representations of $G$. We follow  (\cite{bdk}, Section 2).  We continue with assumption from the first paragraph of the introduction.

We fix a minimal parabolic subgroup $P_0$, its Levi decomposition  $P_0=M_0U_0$, and, as usual related to these choices,  we fix a set of 
standard parabolic subgroups $P=MU$, where $U$ is the unipotent radical, and $M$ is the Levi subgroup such that $M_0\subset M$, $P=MP_0$. We call $M$ a standard Levi subgroup.
(See the text immediately  before the statement of Proposition \ref{intr-m-5}.)

Following   (\cite{bdk}, 2.1), we call the pair $(M, \rho)$, where $M$ is a standard Levi subgroup and $\rho$ an irreducible supercuspidal representation of $M$, a cuspidal pair  of
$G$. Let $\Theta(G)$
be the set of all cuspidal pairs up to a conjugation by $G$.  We write $[M, \rho]$  for the $G$--orbit in $\Theta(G)$ of $(M, \rho)$. A point in $\Theta(G)$ is called infinitesimal
character of $G$.
If we write $\theta\in \Theta(G)$ in the form $\theta=[M, \rho]$, then we say that infinitesimal character $\theta$ is determined by the cuspidal pair $(M, \rho)$.

Let $\Psi(M)$  be the group of all unramified characters if $M$. It has a natural structure of a complex algebraic torus.
A connected component $\Theta$ in $\Theta(G)$ determined by the cuspidal pair $(M, \rho)$ is the image of the map $\Psi(M)\longrightarrow \Theta(G)$ given by
$\psi\longmapsto [M, \psi \rho]$. The set $\Theta$ has a natural structure of a complex affine variety given by quotient of $\Psi(M)$ by a finite group.
We denote by  $\mathcal Z (\Theta)$ its algebra of regular functions.

We have 
$$
\Theta(G)=\cup_\Theta \Theta \ \ \text{(disjoint union}).
$$

By standard theory, given  $\theta\in \Theta(G)$, all parabolically induced representations $\Ind_{P}^G(\rho)$ (normalized induction) have the same
semi--simplifications in the Grothendieck group of all
finite length smooth $G$--modules when cuspidal pairs $(M, \rho)$ range over $G$--orbit $\theta$, and $P=MP_0$. Different points in $\Theta(G)$ determine disjoint
sets of irreducible $G$--modules (after semi--simplification).

Let $\sigma$ be an irreducible smooth $G$--module. Then, by remarks in the previous paragraph, there exists a unique $\theta\in \Theta(G)$  such that if we write 
$\theta=[M, \rho]$, then $\sigma$ is an irreducible subquotient of $\Ind_{P}^G(\rho)$.  We call $\theta$ the infinitesimal character of $\sigma$, and write $\theta=inf. char. (\sigma)$.  

Let $\Irr(G)$ be the set of equivalence classes of smooth irreducible $G$--modules.
Then, the map $\Irr(G)\longrightarrow \Theta(G)$, $\sigma\longmapsto inf. char. (\sigma)$ is finite to one; a preimage
of $\theta\in \Theta(G)$ is the set of all irreducible suqbquotients of  $\Ind_{P}^G(\rho)$ where $\theta=[M, \rho]$, and $P=MP_0$.

\vskip .2in

\begin{Def}\label{bc-1}
We say that the infinitesimal character $\theta\in \Theta(G)$ is in general position if $\Ind_{P}^G(\sigma)$ is irreducible, where $\theta=[M, \rho]$ and $P=MP_0$. 
\end{Def}

Of course, the definition is independent of the choice of the representative $(M, \rho)$.

\begin{Lem} \label{bc-2} Let $\Theta\subset \Theta(G)$ be a connected component.
  Then,  there exists Zariski open set $U\subset \Theta$ such that $\theta \in U$ is in general position. 
  \end{Lem}
\begin{proof} Assume that $\Theta$ is  the image of the map $\alpha: \Psi(M)\longrightarrow \Theta(G)$ given by
  $\alpha(\psi)= [M, \psi \rho]$.

  It is well--known that there exists Zariski open set $U'\subset \Psi(M)$ such that
  $\Ind_{P}^G(\psi \rho)$ is irreducible for $\psi\in U'$. For example, this set is obtained if we take all $\psi$ such that
  the normalized Jacquet module with respect to $P$ has different central characters, and the long--intertwining operator is regular and an isomorphism.
  That kind of standard and well--known considerations can be easily extracted from  \cite{muic-1}. The proof is also contained in (\cite{renard}, Th\' eorem\` e VI.8.5)

  As we recalled above, the set $\Theta$ has  natural structure of a complex affine variety given by quotient of $\Psi(M)$ by a finite group.
  This implies that canonical regular map $\alpha: \Psi(M)\longrightarrow\Theta$ is a finite regular map between affine alebraic varieties. In particular, the image of any closed set
  is closed (\cite{sha}, Chapter 5, Section 3, Corollary).  The required open set is
  $$
  U=\Theta \setminus \alpha\left(\Psi(M)-U'\right)\subset \alpha(U').
  $$
  \end{proof}

Following (\cite{bdk}, Section 2), we let 
$$
\mathcal Z(G)=\prod_\Theta \ \mathcal Z (\Theta).
$$
This $\mathbb C$--algebra can be interpreted as an algebra of regular functions on the affine variety $\Theta(G)$ with infinitely many connected components $\Theta$.
The ideal 
$$
\mathcal Z(G)^0 \overset{def}{=}\oplus_\Theta \ \mathcal Z(\Theta)
$$
is a proper  ideal in $\mathcal Z(G)$.  One can easily check that $\mathbb C$--algebra homomorphisms $\mathcal Z(G)\longrightarrow \mathbb C$ non--trivial on $\mathcal Z(G)^0$
are exactly evaluations at points in $\Theta(G)$.

\vskip .2in

We recall from (\cite{Be}, 2.13) the following result:

\begin{Thm} \label{bc-3} 
For each smooth $G$--module $V$ there exists a homomorphism $\mathcal Z(G)\longrightarrow End_G(V)$ of $\mathbb C$--algebras
such that the following holds: 

\begin{itemize}
\item[(C-1)] if $V$ is irreducible, then the action of $z\in \mathcal Z(G)$ is given by $z=inf. char. (V)(z)1_V$;
\item[(C-2)] we have $\Hom_G(V, \ V')\subset Hom_{\mathcal Z(G)}(V, \ V')$ for all  smooth $G$--modules $V$ and $V'$.
\end{itemize}
The properties (i) and (ii) determine the system of $\mathbb C$--algebra homormorphisms $\mathcal Z(G)\longrightarrow End_G(V)$, where $V$ ranges over smooth $G$--modules, uniquely.
\end{Thm}

\vskip .2in
Next, we recall the following result (\cite{Be}, Proposition 3.3):

\begin{Thm} \label{bc-4}   
A  finitely generated smooth representation $V$  is $\mathcal Z(G)$--admissible i.e.,
  for each open--compact subgroup $L\subset G$ the space of 
  $L$--invariants $V^L$ is finitely generated $\mathcal Z(G)$--module.
\end{Thm}

\vskip .2in
Finally, we recall the Decomposition theorem (\cite{Be}, 2.10):

\begin{Thm} \label{bc-5}   Let $V$ be a smooth $G$--module. Let $1_\Theta \in \mathcal Z(\Theta)$ be the identity for each connected component $\Theta$. Then, 
  $1_\Theta$ act on $V$ as a  projector on a $G$--submodule denoted by $V_\Theta$. We  have
  $$
  V=\oplus_\Theta \ V_\Theta
  $$
Moreover, for any open compact subgroup $L\subset V$, there exists only finitely many connected components $\Theta$ such that $V^L_\Theta\neq 0$.
\end{Thm}

\vskip .2in
We end this section with the following two well--known observations:

\begin{Cor} \label{bc-600}  Assume that $V$ is an irreducible  smooth $G$--module. Then, there exists a unique connected component $\Theta$ such that $V=V_\Theta$. Moreover, we have
  $inf. char. (V)\in \Theta$
\end{Cor}
\begin{proof} Since $V$ is irreducible, the first claim is obvious from Theorem \ref{bc-5}.  This implies that $1_\Theta$ acts as identity on $V$. In particular,
  $$
  inf. char. (V)(1_\theta)=1
  $$
  by (C-1) in Theorem \ref{bc-3}.  
This implies the second claim.  
\end{proof}

\vskip .2in

\begin{Cor} \label{bc-6}  Let $\Theta$ be a connected component. Then, the functor $V\longmapsto V_\Theta$ is exact functor from the category of all smooth $G$--modules into the
  same category.
\end{Cor}
\begin{proof} Consider the short exact sequence
  $$
  \begin{CD}
    0@>>> V@>\alpha >> W @>\beta>> U @>>>0 
  \end{CD}
  $$
  of smooth $G$--modules maps. By Theorem \ref{bc-3}  (C-2), we have 
  $$
  \left(\text{the action of $1_\Theta$ on $W$}\right) \circ \alpha= \alpha \circ  \left(\text{the action of $1_\Theta$ on $V$}\right),
  $$
  and
  $$
  \left(\text{the action of $1_\Theta$ on $U$}\right)\circ \beta= \alpha \circ  \left(\text{the action of $1_\Theta$ on $W$}\right).
  $$

  Since $1_\Theta$ acts as a projection, i.e., $1^2_\Theta= 1_\Theta$, this implies that we have the following sequence of $G$--modules maps:
   $$
  \begin{CD}
    V_\Theta @>\alpha|_{V_\Theta} >> W_\Theta @>\beta|_{W_\Theta}>> U_\Theta.
  \end{CD}
  $$
  It is obvious that $\alpha|_{V_\Theta}$ is injective and that $\beta|_{W_\Theta}\circ \alpha|_{V_\Theta}=0$.  Again, since $1_\Theta$ acts as a projection,
  one sees that the kernel of $\beta|_{W_\Theta}$ is equal  to the image of $\alpha|_{V_\Theta}$ as well as that $\beta|_{W_\Theta}$ is surjective.
  For example, if $\beta|_{V_\Theta}(w)=0$, then there exists
  $v\in V$ such that $w=\alpha(v)$ since the original sequence is exact. Then,
  $$
  w=1_\Theta.w= 1_\Theta.\alpha(v)=\alpha\left(1_\Theta.v\right)= \alpha|_{V_\Theta}\left(1_\Theta.v\right).
  $$
  \end{proof}

\section{The proof of Theorem \ref{intr-m-1}}\label{pmr}

As we already mentioned in the introduction, it is observed in (\cite{muic}, Theorem 3-9) (see Lemma \ref{prelim-1} (v)) that a smooth $G$--module
$$
  \mathcal W(I, L)\overset{def}{=} \mathcal H \star \epsilon_L/\mathcal H\star I
  $$
has a unique maximal proper subrepresentation, and the corresponding quotient is
$\mathcal V(I, L)$. Moreover, the canonical projection $\mathcal W(I, L)^L\longrightarrow \mathcal V(I, L)^L$ is an isomorphism of $\mathcal H_L$--modules.

It remains to prove the most difficult part: $\mathcal W(I,  L)$ is an admissible smooth $G$--module of finite length,
and all irreducible subquotients have the same infinitesimal character. For this, we use results on Bernstein center (\cite{bdk}, \cite{Be})
stated in  Section \ref{bc}.

First, we use the Decomposition theorem (see Theorem \ref{bc-5}). As a result, we obtain the decomposition
as smooth $G$--modules
$$
\mathcal W(I, L)= \oplus_\Theta \ \mathcal W(I, L)_\Theta, 
$$
and as $\mathcal H_L$--modules
$$
\mathcal W(I, L)^L= \oplus_\Theta \ \mathcal W(I, L)^L_\Theta.
$$
But as we recalled above,  $\mathcal W(I, L)^L$ is isomorphic to $\mathcal V(I, L)^L$ as a $\mathcal H_L$--module. Hence,
$\mathcal W(I, L)^L$ is irreducible $\mathcal H_L$--module. Therefore, there exists
a unique connected component $\Theta$ such that
$$
\mathcal W(I, L)^L=  \mathcal W(I, L)^L_\Theta.
$$
Since
\begin{equation}\label{pmr-1}
  \epsilon_L+ \mathcal H\star I \in \mathcal W(I, L)^L
\end{equation}
generates $\mathcal W(I, L)$ as a $\mathcal H$--module (see Lemma \ref{prelim-1} (v)), we see that
$$
\mathcal W(I, L)=  \mathcal W(I, L)_\Theta.
$$
Since irreducible smooth module $\mathcal V(I, L)$  is a quotient of $\mathcal W(I, L)$, we must have by Corollary \ref{bc-6}
$$
\mathcal V(I, L)=  \mathcal V(I, L)_\Theta.
$$

This implies that  the  infinitesimal character of $\mathcal V(I, L)$ , say
$$
\theta=[M, \rho],
$$
must belong to the connected component $\Theta$ (see Corollary \ref{bc-600}).  Next, by Theorem \ref{bc-3} (C-1),  we must have that  $\mathcal Z(G)$ acts as
as a character $\theta$ i.e.,
$$
z=\theta(z)1_{\mathcal V(I, L)}, \ \ z\in \mathcal Z(G).
$$

Using the isomorphism of  $\mathcal W(I, L)^L\simeq \mathcal V(I, L)^L$ of $\mathcal H_L$--modules, we see that
$\mathcal Z(G)$ acts as a character $\theta$ on $\mathcal W(I, L)^L$. 
But,  $\mathcal W(I, L)$ is as a smooth $G$--module generated by a class in (\ref{pmr-1}). Hence,
$\mathcal Z(G)$ acts as  a character $\theta$ on $\mathcal W(I, L)$. This implies that every irreducible subqoutient of $\mathcal W(I, L)$ has infinitesimal character $\theta$.

Finally, we prove that $\mathcal W(I,  L)$ is an admissible smooth $G$--module of finite length. Let $L'\subset G$ be an open compact subgroup. Then, by Theorem \ref{bc-4}, 
$\mathcal W(I, L)^{L'}$ is finitely generated as $\mathcal Z(G)$--module.  Since  $\mathcal Z(G)$ acts on  $\mathcal W(I, L)$ as a character $\theta$, we see that
$\dim_{\mathbb C} \mathcal W(I, L)^{L'}<\infty$. Since $L'$ is arbitrary, we obtain that  $\mathcal W(I,  L)$ is  admissible. Finally, it has finite length since every finitely generated
admissible $G$--module has finite length (see \cite{cas}, theorem 6.3.10, or  \cite{BZ}, Theorem 4.1 for $GL_N$). This completes the proof of  Theorem \ref{intr-m-1}.

\section{Computation of Certain Jacquet modules and Proof of Proposition \ref{intr-m-5}}\label{cjm}

We maintain the notation from the Introduction and Section \ref{bc}. We recall the notion of a normalized Jacquet module.  Let $P=MU$  be a standard parabolic subgroup of $G$.
Let $V$ be a smooth $G$--module. Then, $\mathbb C$--span, say $V(U)$,  of all $v- u.v$, $v\in V$ and $u\in U$, is $M$--invariant. Therefore, the quotient $V/V(U)$ is canonically
smooth $M$--module. The corresponding normalized Jacquet module $r_{M, G}\left(V\right)$ is the space
$V/V(U)$ under the action of $M$ given by
$$
m.\left(v+ V(U)\right)= \delta^{-1/2}_{P}(m) m.v + V(U), \ \ v\in V, \ m\in M.
$$
Here $\delta_P$ is usual modular character of $P$.

Let us fix some Haar measure on $U$, for example normalized with $\int_{U\cap K} du=1$, where $K$ is a hyperspecial maximal compact subgroup
(see the text before the statement of  Proposition  \ref{intr-m-5}).
The space $V(U)$, can also be characterized as follows (see \cite{cas}, 3.2): $v\in V(U)$ if and only if $\int_{L_U} u.v\ du=0$ for some open compact subgroup $L_U\subset U$.  

\vskip .2in

We begin with

\begin{Lem}\label{cjm-1} 
  Let $P=MU$  be a standard parabolic subgroup of $G$. Let $f\longmapsto f_P$ be the constant term map $C_c^\infty(G)\longrightarrow C_c^\infty(U\backslash G)$: $f_p(x)=\int_U f(ux) du$. 
  Let $V\subset C_c^\infty(G)$ be a smooth $G$--submodule under the left translation $l$. Then, $r_{M, G}(V)$ is the image of $V$ under the constant term map, and the action of $m\in M$ is given by
  $\delta^{1/2}_P(m) l(m)$.
  \end{Lem} 
\begin{proof} It is obvious that the restriction of the constant term map factorizes through $V/V(U)$. It remains to prove that $f_P$ restricted to $V$ has kernel exactly
  $V(U)$. So, let $f\in V$ such that
  
\begin{equation} \label{cjm-2}
  f_P(x)=\int_{U} f(ux) du=0,  \ \ x\in G.
  \end{equation}

Then, by above recalled characterization of $V(U)$, we must prove that there exists an open compact subgroup $L_U\subset U$ such that
  \begin{equation} \label{cjm-3}
  \int_{L_U} f(ux) du= \int_{L_U} f(u^{-1}x) du= \int_{L_U} l(u)f(x) du=0,  \ \ x\in G.
  \end{equation}

  Let $L\subset K$ be a normal  open compact subgroup such that $f$ is right invariant under $L$. We fix a decomposition
  $$
  K=\cup_{i=1}^l k_iL= \cup_{i=1}^l L k_i , \ \ \text{(disjoint union)}.
  $$
  By Iwasawa decomposition $P=UMK$,  the function $f_P$ is determined by its values on the sets $Mk_i$, $i=1, \ldots, l$.  Let
  $\Omega$ be the (compact) support of $f$. We may assume that $\Omega=\Omega\cdot L$. For any $i=1, \ldots, l$, we define a compact set
  $\Omega_{M, i}$ as the image  of $\Omega\cdot k^{-1}_i\cap P$ of the projection of $P$ onto $M\simeq U\backslash P$. Now, for any $i=1, \ldots, l$,
  there exists a compact subset  $\Omega_{U, i}\subset U$ such that if for $u\in U$, 
  $ f(umk_i)\neq 0$, for some  $m\in M$, then $u\in  \Omega_{U, i}$.

  But since $U$ contains arbitrarily large open compact subgroups, there exists an open compact subgroup $L_U\subset U$ such that
   \begin{equation} \label{cjm-4}
  \cup_{i=1}^l \ \Omega_{U, i}\subset L_U.
   \end{equation}
   
 Thus, we have
 \begin{equation} \label{cjm-5}
  \text{if for $u\in U$,} \ \   f(umk_i) \neq 0, \ \text{for some $m\in M$ and $i$, then $u\in L_U$.} 
 \end{equation}
 
  Now, we compute using (\ref{cjm-2}) and (\ref{cjm-5})
  \begin{equation} \label{cjm-6}
 \int_{L_U} f(umk_i)du=  \int_U f(u mk_i)du= f_P(mk_i)=0,
  \end{equation}
 for all $m\in M$, and all $i=1, \ldots, l$.

 Let  $x\in G$. Then, we can write $x=umk_il$, where $u\in U, m\in M, l\in L$ for some $i$. Since $f$ is right invariant under $L$, we may assume that $l=1$. 
 Now, for $u'\in L_U$, by (\ref{cjm-5}),  we have that
 $f(u'umk_i)\neq 0$  implies $u'u\in L_U$, and consequently $u\in L_U$.  Thus, if $u\in L_U$, then,  using (\ref{cjm-6}), we have the following:
 $$
 \int_{L_U} f(u'x) du'=\int_{L_U} f(u'mk_i) du'=0
 $$
 proving (\ref{cjm-3}).  If $u\not\in L_U$, then $f(u'umk_i)= 0$ for all $u'\in L_U$. Again, this implies
 $$
 \int_{L_U} f(u'x) du'=0.
 $$

 The action of $M$ on $r_{M, G}(V)$ is given by
 $$
 m.f_P(x)= \delta^{-1/2}_P(m) \int_U \ f(m^{-1}ux) du= \delta^{1/2}_P(m) \int_U \ f(u m^{-1}x) du= \delta^{1/2}_P(m) l(m)f_P(x),
 $$
 for $x\in G$. \end{proof}

\vskip .2in
Let $^0\!\!M_0$ be as usual the intersection of the kernels of all characters $|\chi|_k$, where $\chi$ ranges over all rational characters
$\chi: M_0\longrightarrow k^\times$. Here $|\ |_k$ is the norm
of $k$.  Since $P_0$ is a minimal $k$--parabolic subgroup of $G$, we have (see \cite{cartier})
$$
^0\!\!M_0=M_0\cap K.
$$

\vskip .2in
Lemma \ref{cjm-1} implies the following result: 

\begin{Cor}\label{cjm-7}  Let $\delta_0$ be the modular character of $P_0$. Assume that $I$ is a maximal left ideal in $\mathcal H_K$. We define $J$ to be the
  $\mathbb C$--span of all $\delta^{1/2}_0(m_0)l(m_0)f_P$, where $f\in I$, and $m_0\in M_0$.  Then, we have the following isomorphism $M_0$--modules:
  $$
  r_{M_0, G}\left(\mathcal W(I, K)\right)\simeq C_c^\infty(M_0/ ^0\!\!M_0)/J.
  $$
\end{Cor}
\begin{proof} By Lemma \ref{cjm-1}, we have
  $$
  r_{M_0, G}\left(\mathcal H\star \epsilon_K\right)= r_{M_0, G}\left(C_c^\infty(G/K)\right)= C_c^\infty(U\backslash  G /K)= C_c^\infty(M_0/ ^0\!\!M_0),
  $$
  with the action of $m_0\in M_0$ given by  $\delta^{1/2}_0(m_0) l(m_0)$. Also,  its submodule $\mathcal H \star I$, generated by $I$, satisfies
  $$
  r_{M_0, G}\left(\mathcal H\star I\right)= J.
  $$
  Now, the exactness of Jacquet modules (\cite{cas}, Proposition 3.2.3) implies the claim:
  $$
  r_{M_0, G}\left(\mathcal W(I, K)\right)\simeq  r_{M_0, G}\left(\mathcal H\star \epsilon_K\right)/ r_{M_0, G}\left(\mathcal H\star I\right)\simeq
  C_c^\infty(M_0/ ^0\!\!M_0)/J.
  $$
\end{proof}

\vskip .2in
Now, we compute the Jacquet module from Corollary \ref{cjm-7} using some commutative algebra via Satake isomorphism. For this, we recall some more results
from \cite{cartier}.  The group of all unramified characters of $M_0$ (see beginning of Section \ref{bc}) is given by
$$
\Psi(M_0)=\Hom_{\mathbb Z}\left(M_0/ ^0\!\!M_0, \ \mathbb C^\times\right).
$$
It is also obvious that the group $M_0/ ^0\!\!M_0$ is commutative. Therefore,  under the convolution normalized by $\int_{^0\!\!M_0} dm=1$,
$$
\mathcal A \overset{def}{=}C_c^\infty(M/ ^0\!\!M)
$$
is an associative algebra with identity $1_{^0\!\!M_0}$ isomorphic to  the group  $\mathbb C$--algebra of $M_0/ ^0\!\!M_0$. 

Since $M_0/ ^0\!\!M_0$ is a finitely generated free Abelian group, the algebra is finitely generated.  The Weyl group
$$
W=N_{G}(A_0)/ Z_G(A_0)=  N_{G}(A_0)\cap K/ ^0\!\!M_0
$$
acts by conjugation on $M_0$, $^0\!\!M_0$, and $M_0/^0\!\!M_0$, and consequently on
$\mathcal A$.

The subalgebra of all $W$--invariants $\mathcal A^W$ of $\mathcal A$ is the image of the algebra $\mathcal H_K$ under the Satake isomorphism
$$
Sf(m)=\delta_0^{-1/2}(m)f_{P_0}(m), \ \ m\in M_0,
$$
using the notation of Lemma \ref{cjm-1}. In particular,
there is a one--to--one correspondence $I\leftrightarrow \mathfrak m_I$  between   maximal (left) ideals in $\mathcal H_K$, and in $\mathcal A^W$.

\vskip .2in
Now, we make Corollary \ref{cjm-7} explicit. We have the following:

\begin{Cor}\label{cjm-8000}  Assume that $I\subset \mathcal H_K$ is a maximal left ideal. Put  $\mathfrak m=\mathfrak m_I$. Then, 
  \begin{align*}
  &  r_{M_0, G}\left(\mathcal W(I, K)\right)\simeq \mathcal A /\mathfrak m \mathcal A \\
  &  r_{M_0, G}\left(\mathcal H_K/ I\right) \simeq \mathcal A^W /\mathfrak m \mathcal A^W,
\end{align*}
where the action on the right is just the usual left translation twisted by $\delta^{1/2}_0$.
\end{Cor}

\vskip .2in 
We also recall that for $\chi\in \Psi(M_0)$ defines a $\mathbb C$--algebra homomorphism $\mathcal A\longrightarrow \mathbb C$ (\cite{cartier}, page 148) given by
$$
f\longmapsto \int_M f(m)\chi(m) dm.
$$
Two unramified characters define the same $\mathbb C$--algebra homomorphism $\mathcal A^W\longrightarrow \mathbb C$ if and only if are  $W$--conjugate.
In fact, $\mathcal A$ is the algebra of regular functions on
complex affine variety $\Psi(M_0)$ and  $\mathcal A^W$ is the algebra of regular functions on the affine variety of the $W$--orbits. In particular, 
maximal ideals in $\mathcal A^W$  are in one--to--one correspondence with  $W$--orbits of unramified characters in $\Psi(M_0)$: $\mathfrak m\leftrightarrow \mathcal O_{\mathfrak m}$.
Thus, there exists a one--to--one correspondence between maximal ideals in $\mathcal H_K$, and   $W$--orbits of unramified characters in $\Psi(M_0)$: $I\leftrightarrow \mathcal O_{\mathfrak m_I}$.
We say that $I$ (or $\mathfrak m_I$) is regular if $\# \mathcal O_{\mathfrak m_I}=\# W$. This is equivalent to the fact that orbit contains $W$--regular (i.e., 
  its stabilizer in $W$ is trivial) unramified character.

\vskip .2in
We have the following:

\vskip .2in 
\begin{Prop}\label{cjm-8} Assume that $I\subset \mathcal H_K$ is a regular  maximal left ideal. Let  $\chi\in \mathcal O_{\mathfrak m_I}$. 
Then, $\chi$ is $W$--regular, and we have an isomorphism of $M_0$--modules
  $$
  r_{M_0, G}\left(\mathcal W(I, K)\right)\simeq  r_{M_0, G}\left(\Ind_{P_0}^G(\chi)\right)\simeq \oplus_{w\in W} \ w(\chi).
  $$
\end{Prop}
\begin{proof} The isomorphism
  $$
  r_{M_0, G}\left(\Ind_{P_0}^G(\chi)\right)\simeq \oplus_{w\in W} \ w(\chi)
  $$
  is well--known (see \cite{cas}, Proposition 6.4.1). It remains to prove

  \begin{equation}\label{cjm-80} 
  r_{M_0, G}\left(\mathcal W(I, K)\right)\simeq \oplus_{w\in W} \ w(\chi).
  \end{equation}

  First, let $\pi$ be a unique irreducible subqoutient of $\Ind_{P_0}^G(\chi)$ which is $K$--spherical. Then, there exists a subset $S\subset W$ such that
   $$
  r_{M_0, G}\left(\pi\right)\simeq \oplus_{w\in S} \ w(\chi).
  $$
  Since   $\chi$ is $W$--regular, for any other irreducible subquotient $\pi'$ of $\Ind_{P_0}^G(\chi)$, $w(\chi)$, $w\in S$, does not show up in  $r_{M_0, G}\left(\pi'\right)$.

Next, by above description of $r_{M_0, G}\left(\mathcal W(I, K)\right)$ (see Corollary \ref{cjm-8000}), and  general Lemma \ref{cjm-9}, there exists a positive integer $L$ such that the 
  semi--simplification of  $r_{M_0, G}\left(\mathcal W(I, K)\right)$ is given by
  $$
  r_{M_0, G}\left(\mathcal W(I, K)\right)= L\cdot \sum_{w\in W} \ w(\chi)
  $$
  in the Grothendieck group of finite  length admissible modules of $M_0$.

  By  Theorem \ref{intr-m-1}, $\mathcal W(I, K)$ has finite length, all of its irreducible subqoutients are among those of $\Ind_{P_0}^G(\chi)$, and $\pi$ appears
  with multiplicity one in its composition series.

  This implies   $L=1$ because $w(\chi)$, $w\in S$, show up only in  $r_{M_0, G}\left(\pi\right)$, and $\pi$ appears with
  multiplicity one in the composition series of  $\mathcal W(I, K)$.

  Finally, since $\chi$ is regular, we obtain (\ref{cjm-80}). 
\end{proof}

\vskip .2in
We include the following  general lemma. Before we state the lemma, we introduce some notation.  Let $X$ be a complex affine variety, and
let $\mathcal A$ be its algebra of regular functions.
Assume that a finite group $W$ acts on $X$ as a group of regular (algebraic)
transformations. Let $Y$   be the affine variety of $W$--orbits. Then, by standard theory,  its algebra of regular functions is $\mathcal A^W$.

\begin{Lem}\label{cjm-9}
    Let $x\in X$ be a point such that the orbit $W.x$ has $\# W$ --distinct elements (a generic orbit). Let
  $\mathfrak m_w$ be the maximal ideal that corresponds to $w.x$ for $w\in W$. Then, $$\mathfrak m\overset{def}{=}
  \mathfrak m_w\cap \mathcal A^W$$ is independent of $w\in W$. As a $\mathcal A$--module, $\mathcal A/ \mathfrak m \mathcal A$ has filtration by irreducible $\mathcal A$--modules
  $ \mathcal A/ m_w$, $w\in W$, where their multiplicity in the composition series is independent of $w$. 
  \end{Lem}
\begin{proof}  We use Nullstellensatz, and results about  the support of finitely generated modules (see \cite{lang}, Chapter X, Section 2). Put
  $$
  M\overset{def}{=}   \mathcal A/ \mathfrak m \mathcal A.
  $$
  This is a finitely generated $\mathcal A$--module which by the general theory has filtration by $\mathcal A/\mathfrak p$ where $\mathfrak p$ ranges over primes in the support $\supp{(M)}$ of $M$.
  By definition,  $\mathfrak p\in \supp{(M)}$ if and only if the localization $M_{\mathfrak p}$ satisfies $M_{\mathfrak p}\neq 0$. The radical of the annihilator of $M$ is given by the following
  general formula:
  \begin{equation}\label{cjm-10}
  \text{rad}\left(\text{ann}\left(M\right)\right)=\cap_{\mathfrak p\in \supp{(M)}}\mathfrak p. 
      \end{equation}
   It is obvious that
  $$
   \text{ann}\left(M\right)= \mathfrak m \mathcal A, 
   $$
   and by Nullstellensatz
     \begin{equation}\label{cjm-11}
   \text{rad}\left( \mathfrak m \mathcal A\right)=\cap_{w\in W}\mathfrak m_w.
   \end{equation}
   This implies that we have an epimorphism of $\mathcal A$--modules
    $$
   M=\mathcal A/ \mathfrak m \mathcal A \longrightarrow \mathcal A/ \cap_{w\in W}\mathfrak m_w   \longrightarrow  0.
   $$ 
   Next, by the Chinese remainder theorem, this gives the following epimorphism of  $\mathcal A$--modules:
   $$
   M\longrightarrow \oplus_{w\in W} \ \mathcal A/  \mathfrak m_w \longrightarrow  0.
   $$
  Thus, 
   $$
   \left\{\mathfrak m_w; \ \ w\in W\right\}\subset  \supp{(M)}.
   $$
   But (\ref{cjm-10}) and (\ref{cjm-11}) imply
   $$
   \cap_{w\in W}\mathfrak m_w= \cap_{\mathfrak p\in \supp{(M)}}\mathfrak p.
   $$
   Then, for each $\mathfrak q\in \supp{(M)}$, we have
   $$
   \prod_{w\in W}\mathfrak m_w \subset \cap_{w\in W}\mathfrak m_w \subset \mathfrak q.
   $$
   Then,  the definition of a prime ideal implies that
   $$
   \mathfrak m_w \subset \mathfrak q,
   $$
   for some $w$.    Since $\mathfrak m_w$ is a maximal ideal, we obtain $\mathfrak q=\mathfrak m_w$. Thus, we
   have
   $$
   \supp{(M)}=   \left\{\mathfrak m_w; \ \ w\in W\right\}.
   $$

   Finally, this implies that $M$  has filtration by irreducible $\mathcal A$--modules
  $ \mathcal A/ m_w$, $w\in W$, where their multiplicity in the composition series is independent of $w$ since $W$--permutes composition factors. 
   \end{proof}

\end{document}